\documentclass{article}
\usepackage{graphics}
\usepackage{epsfig}
\usepackage{amssymb}
\usepackage{amsmath}

\newcommand{\A}{{\mathcal A}}
\newcommand{\B}{{\mathcal B}}
\newcommand{\C}{{\mathcal C}}
\newcommand{\D}{{\mathcal D}}

\newtheorem{theorem}{Theorem}[section]
\newtheorem{defn}[theorem]{Definition}
\newtheorem{lemma}[theorem]{Lemma}

\newtheorem{corollary}[theorem]{Corollary}

\newtheorem{proposition}[theorem]{Proposition}

\newenvironment{proof}
    {\pagebreak[1]{\narrower\noindent {\bf
    Proof:\quad\nopagebreak}}}{$\hfill\square$}

\begin{document}

\author{ \textbf{Wesley Calvert} \\
Department of Mathematics \\
University of Notre Dame\\
wcalvert@nd.edu\\
\textbf{Douglas Cenzer} \\
Department of Mathematics \\
University of Florida\\
cenzer@math.ufl.edu\\
\textbf{Valentina Harizanov} \\
Department of Mathematics \\
George Washington University\\
harizanv@gwu.edu\\
\textbf{Andrei Morozov} \\
Sobolev Institute of Mathematics\\
morozov@math.nsc.ru}
\title{Effective Categoricity of Equivalence Structures\\
FINAL DRAFT}
\maketitle

\section{Introduction}

We consider only countable structures for computable languages with
universe $\omega$. We identify
sentences with their G\"{o}del codes. The \emph{atomic} \emph{diagram} of a
structure $\mathcal{A}$ for $L$ is the set of all quantifier-free sentences
in $L_{A}$, $L$ expanded by constants for the elements in $A$, which are
true in $\mathcal{A}$. A structure is \emph{computable} if its atomic
diagram is computable. In other words, a structure $\mathcal{A}$ is
computable if there is an algorithm that determines for every
quantifier-free formula $\theta (x_{0},\ldots ,x_{n-1})$ and every
sequence $(a_{0},\ldots ,a_{n-1})\in A^{n}$, whether $\mathcal{A}\vDash \theta
(a_{0},\ldots ,a_{n-1})$. The \emph{elementary}\textit{\ }\emph{diagram} of 
\emph{\ }$\mathcal{A}$ is the set of all sentences of $L_{A}$ that are true
in $\mathcal{A}$. A structure $\mathcal{A}$ is \emph{decidable} if its
elementary diagram is computable. For $n>0$, the $n$-\emph{diagram} of
\emph{\ }$\mathcal{A}$ is the set of all $\Sigma _{n}$ sentences of
$L_{A}$ that are true in $\mathcal{A}$. A structure is
$n$-\emph{decidable} if its $n$-diagram is computable. 

A computable structure $\mathcal{A}$ is \emph{computably categorical} if for
every computable isomorphic copy $\mathcal{B}$ of $\mathcal{A}$, there is a 
\emph{computable} isomorphism from $\mathcal{A}$ onto $\mathcal{B}$. For
example, the ordered set of rational numbers is computably categorical,
while the ordered set of natural numbers is not. Moreover, Goncharov and
Dzgoev \cite{GD} and Remmel \cite{R} independently proved that a computable
linear ordering is computably categorical if and only if it has only
finitely many successors. They also established that a computable Boolean
algebra is computably categorical if and only if it has finitely many atoms
(see also LaRoche \cite{L}). Miller \cite{Miller} proved that no computable
tree of height $\omega $ is computably categorical. Lempp, McCoy, Miller and
Solomon \cite{L-M-M-S} characterized computable trees of finite height that
are computably categorical. Nurtazin \cite{N} and Metakides and Nerode \cite
{MN} established that a computable algebraically closed field of finite
transcendence degree over its prime field is computably categorical.
Goncharov \cite{G1} and Smith \cite{Sm} characterized computably categorical
abelian $p $-groups as those that can be written in one of the following
forms: $(Z(p^{\infty }))^{l}\oplus G$ for $l \in \omega \cup \{\infty \}$
and $G $ finite, or $(Z(p^{\infty }))^{n}\oplus G\oplus (Z(p^{k}))^{\infty }$
, where $n,k\in \omega $ and $G$ is finite. Goncharov, Lempp and Solomon 
\cite{G-L-S} proved that a computable, ordered, abelian group is computably
categorical if and only if it has finite rank. Similarly, they showed that a
computable, ordered, Archimedean group is computably categorical if and only
if it has finite rank.

For any computable ordinal $\alpha $, we say that a computable
structure $\mathcal{A}$ is \emph{$\Delta _{\alpha }^{0}$ categorical}
if for every computable structure $\mathcal{B}$ isomorphic to
$\mathcal{A}$, there is a $\Delta _{\alpha }^{0}$ isomorphism form
$\mathcal{A}$ onto $\mathcal{B}$.  Lempp, McCoy, Miller and Solomon
\cite{L-M-M-S} proved that for every $n\geq 1$, there is a computable
tree of finite height that is $\Delta_{n+1}^{0}$-categorical but not
$\Delta _{n}^{0}$-categorical. We say that $\mathcal{A}$ is
\emph{relatively computably categorical} if for every structure
$\mathcal{B}$ isomorphic to $\mathcal{A}$, there is an isomorphism
that is computable relative to the atomic diagram of
$\mathcal{B}$. Similarly, a computable $\mathcal{A}$ is
\emph{relatively $\Delta _{\alpha }^{0}$ categorical} if for every
$\mathcal{B}$ isomorphic to $\mathcal{A}$, there is an isomorphism
that is $\Delta _{\alpha }^{0}$ relative to the atomic diagram of
$\mathcal{B}$. Clearly, a relatively $\Delta _{\alpha }^{0}$
categorical structure is $\Delta _{\alpha }^{0}$ categorical. We are
especially interested in the case when $\alpha =2$. McCoy \cite{McCoy}
\ characterized, under certain restrictions, all $\Delta _{2}^{0}$
categorical and relatively $\Delta _{2}^{0}$ categorical linear
orderings and Boolean algebras. For example, a computable Boolean
algebra is relatively $\Delta _{2}^{0}$ categorical if \ and only if
it can be expressed as a finite direct sum $c_{1}\vee \ldots \vee
c_{n}$, where each $c_{i}$ is either atomless, an atom, or $1$-atom.
Using an enumeration result of Selivanov \cite{S}, Goncharov \cite{G2}
showed that there is a computable structure that is computably
categorical but not relatively computably categorical. Using a
relativized version of Selivanov's enumeration result, Goncharov,
Harizanov, Knight, McCoy, Miller and Solomon \cite{G-H-K-M-M-S} showed
that for each computable successor ordinal $\alpha $, there is a
computable structure that is $\Delta _{\alpha }^{0}$ categorical, but
not relatively $\Delta_{\alpha}^{0}$ categorical.

\label{sec:background}There are syntactical conditions that are equivalent
to relative $\Delta _{\alpha }^{0}$ categoricity. The conditions involve the
existence of certain families of formulas, that is, certain Scott families.
Scott families come from Scott's Isomorphism Theorem, which says that for a
countable structure $\mathcal{A}$, there is an $L_{\omega _{1}\omega }$
sentence whose countable models are exactly the isomorphic copies of 
$\mathcal{A}$. A \emph{Scott family }for a structure $\mathcal{A}$ is a
countable family $\Phi $ of $L_{\omega _{1}\omega }$ formulas, possibly with
finitely many fixed parameters from $A$, such that:

(i) Each finite tuple in $\mathcal{A}$ satisfies some $\psi \in \Phi $;

(ii)\ If $\overrightarrow{a}$, $\overrightarrow{b}$ are tuples in
$\A$, of the same length, satisfying the \emph{same} formula in $\Phi $, then
there is an automorphism of $\mathcal{A}$ that maps
$\overrightarrow{a}$ to $\overrightarrow{b}$.

A \emph{formally c.e.\ Scott family} is a c.e.\ Scott family
consisting of finitary existential formulas. A \emph{formally $\Sigma
_{\alpha }^{0}$ Scott family} is a $\Sigma _{\alpha }^{0}$ Scott
family consisting of computable $\Sigma _{\alpha }$ formulas. Roughly
speaking, computable infinitary formulas are $L_{\omega _{1}\omega }$
formulas in which the infinite disjunctions and conjunctions are taken
over computably enumerable (c.e.) sets. We can classify computable
formulas according to their complexity in the following way: A
computable $\Sigma _{0}$ or $\Pi _{0}$ formula is a finitary
quantifier-free formula. Let $\alpha $ $>0$ be a computable ordinal. A
computable $\Sigma _{\alpha }$ formula is a c.e.\ disjunction of
formulas $(\exists \overrightarrow{u})\theta
(\overrightarrow{x},\overrightarrow{u})$, where $\theta $ is
computable $\Pi _{\beta }$ for some $\beta <\alpha $. A computable
$\Pi _{\alpha }$ formula is a c.e.\ conjunction of formulas $(\forall
\overrightarrow{u})\theta (\overrightarrow{x},\overrightarrow{u})$,
where $\theta $ is computable $\Sigma _{\beta}$ for some $\beta
<\alpha $. Precise definition of computable infinitary formulas
involves assigning indices to the formulas, based on Kleene's system
of ordinal notations (see \cite{A-K}). The important property of these
formulas is given in the following theorem due to Ash.

\begin{theorem}[Ash]
For a structure $\mathcal{A}$, if $\theta (\overrightarrow{x})$ is a
computable $\Sigma _{\alpha }$ formula, then the set $\{\overrightarrow{a}:
\mathcal{A}\models \theta (\overrightarrow{a})\}$ is $\Sigma _{\alpha }^{0}$
relative to $\mathcal{A}$.
\end{theorem}

\noindent An analogous result holds for computable $\Pi _{\alpha }$ formulas.

It is easy to see that if $\mathcal{A}$ has a formally c.e.\ Scott family,
then $\mathcal{A}$ is relatively computably categorical. In general,
if 
$\mathcal{A}$ has a formally $\Sigma _{\alpha }^{0}$ Scott family,
then 
$\mathcal{A}$ is relatively $\Delta _{\alpha }^{0}$ categorical. Goncharov 
\cite{G2} showed that if $\mathcal{A}$ is 2-decidable and computably
categorical, then it has a formally c.e.\ Scott family. Ash \cite{A} showed
that, under certain decidability conditions on $\mathcal{A}$, if $\mathcal{A}
$ is $\Delta _{\alpha }^{0}$ categorical, then it has a formally $\Sigma
_{\alpha }^{0}$ Scott family. For the relative notions, the decidability
conditions are not needed. Ash, Knight, Manasse and Slaman \cite{A-K-M-S}
and Chisholm \cite{C} independently proved that a computable structure 
$\mathcal{A}$ is relatively $\Delta _{\alpha }^{0}$ categorical iff it has a
formally $\Sigma _{\alpha }^{0}$ Scott family.

Cholak, Goncharov, Khoussainov and Shore \cite{C-G-K-S} gave an example of a
computable structure that is computably categorical, but ceases to be after
naming any element of the structure. Such a structure is not relatively
computably categorical. On the other hand, Millar \cite{Millar} established
that if a structure $\mathcal{A}$ is $1$-decidable, then any expansion
of $\A$ by finitely many constants remains computably categorical.
Khoussainov and Shore \cite{K-S} proved that there is a computably
categorical structure without a formally c.e.\ Scott family whose expansion
by any finite number of constants is computably categorical. A similar
result was established by Kudinov by a quite different method. Using a
modified family of enumerations constructed by Selivanov \cite{S}, Kudinov
produced a computably categorical, $1$-decidable structure without a
formally c.e.\ Scott family.

A structure is \emph{rigid} if it does not have nontrivial
automorphisms. A computable structure is \emph{$\Delta _{\alpha }^{0}$
stable} if every isomorphism from $\mathcal{A}$ onto a computable
structure is $\Delta _{\alpha }^{0}$. If a computable structure is
rigid and $\Delta _{\alpha }^{0}$ categorical, then it is $\Delta
_{\alpha }^{0}$ stable. A \emph{defining family} for a structure
$\mathcal{A}$ is a set $\Phi $ of formulas with one free variable and
a fixed finite tuple of parameters such that:

(i)\ Every element of $A$ satisfies some formula $\psi (x)\in \Phi $;

(ii)\ No formula of $\Phi $ is satisfied by more than one element of $A$.

\noindent A defining family $\Phi $ is 
\emph{formally $\Sigma _{\alpha}^{0}$} 
if it is a $\Sigma _{\alpha }^{0}$ set of computable $\Sigma _{\alpha }$
formulas. In particular, a defining family $\Phi $ is \emph{formally c.e.}
if it is a c.e.\ set of finitary existential formulas. For a rigid
computable structure $\mathcal{A}$, there is a formally $\Sigma _{\alpha
}^{0}$ Scott family iff there is a formally $\Sigma _{\alpha }^{0}$ defining
family.

It is not known whether for a computable limit ordinal $\alpha $, there is a
computable structure that is $\Delta _{\alpha }^{0}$ categorical but not
relatively $\Delta _{\alpha }^{0}$ categorical (see \cite{G-H-K-M-M-S}). It
is also not known whether for any computable successor ordinal $\alpha $,
there is a rigid computable structure that is $\Delta _{\alpha }^{0}$
categorical but not relatively $\Delta _{\alpha }^{0}$ categorical. Another
open question is whether every $\Delta _{1}^{1}$ categorical computable
structure must be relatively $\Delta _{1}^{1}$ categorical 
(see \cite{G-H-K-S}).

In Section 2, we investigate algorithmic properties of computable
equivalence structures, their equivalence classes, and their characters. in
Section 3, we examine effective categoricity of equivalence structures. We
characterize the computably categorical equivalence structures and show that
they are all relatively computably categorical. That is, $\mathcal{A}$ is
computably categorical if and only if the following two conditions are
satisfied: 

(i) There is an upper bound on the size of the finite equivalence classes of 
$\mathcal{A}$;

(ii) There is at most one cardinal $k$ such that $\mathcal{A}$ has
infinitely many equivalence classes of size $k$. 

In Section 3, we characterize the relatively $\Delta _{2}^{0}$ categorical
equivalence structures as those with either finitely many infinite
equivalence classes or with an upper bound on the size of the finite
equivalence classes. We also consider the complexity of isomorphisms
for structures $\mathcal A$ and $\mathcal B$ such that both $Fin^{\A}$
and $Fin^{\B}$ are computable or $\Delta^0_2$. 
Finally, we show that every computable equivalence
structure is relatively $\Delta _{3}^{0}$ categorical.

\section{Computable Equivalence Structures}

An equivalence structure $\mathcal{A}=(A,E^{A})$ consists of a set with a
binary relation that is reflexive, symmetric and transitive. An equivalence
structure $\mathcal{A}$ is \emph{computable} if $A$ is a computable subset
of $\omega $ and $E$ is a computable relation. If $A$ is an infinite set
(which is usual), we can assume, without loss of generality, that
$A=\omega $. 
The $\mathcal{A}$-equivalence class of $a\in A$ is 
\begin{equation*}
\lbrack a]^{\mathcal{A}}=\{x\in A:xE^{\mathcal{A}}a\}\text{.}
\end{equation*}%
We generally omit the superscript $^{\mathcal{A}}$ when it can be inferred
from the context.

We will proceed from the simpler structures, which are computably
categorical, to the more complicated structures, which are $\Delta^0_3$
categorical but not $\Delta^0_2$ categorical.

\begin{defn}
The character $\chi(\mathcal{A})$ of an equivalence relation $\mathcal{A}$
is the set $\{\langle k,n \rangle:\ {\mathcal{A}}\ \text{has}\ \text{at
least $n$ equivalence classes of size}\ k \}$.
\end{defn}

We say that $\mathcal{A}$ has \emph{bounded} \emph{character} if there is
some finite $K$ such that all finite $A$-equivalence classes have size at
most $K$.

The following lemmas will be needed.

\begin{lemma}
\label{l1} For any computable equivalence structure $\mathcal{A}$:

\begin{enumerate}
\item[(a)] $\{\langle k,a \rangle: card([a]^{\mathcal{A}}) \leq k\}$ is a 
$\Pi^0_1$ set and $\{\langle k,a \rangle: card([a]^{\mathcal{A}}) \geq k\}$
is a $\Sigma^0_1$ set.

\item[(b)] $Inf^{\mathcal{A}} = \{a: [a]^{\mathcal{A}}\ \text{is infinite}\}$
is a $\Pi^0_2$ set and $Fin^{\mathcal{A}} = \{a: [a]^{\mathcal{A}}\ \text{is
finite}\}$ is a $\Sigma^0_2$ set.

\item[(c)] $\chi(A)$ is a $\Sigma^0_2$ set.
\end{enumerate}
\end{lemma}

\begin{proof}
(a) The condition $card([a]^{\mathcal{A}})\leq k$ holds if and only if the
statement 
\begin{equation*}
(\forall x_{1})\dotsb (\forall x_{k+1})\left( {}\right. (x_{1}Ea\ \&\ \dotsb
\ \&\ x_{k+1}\ Ea)\Rightarrow \bigvee_{i,j\neq k+1}x_{i}=x_{j}\left.
{}\right) 
\end{equation*}
is satisfied.

(b) We have $a\in Fin^{\mathcal{A}}$ if and only if 
\begin{equation*}
(\exists k)[card([a]^{\mathcal{A}})=k]
\end{equation*}
and $a\in Inf^{\mathcal{A}}$ if and only if $a\notin Fin^{\mathcal{A}}$.

(c) We have $\langle k,n\rangle \in \chi (\mathcal{A})$ if and only if 
\begin{equation*}
(\exists x_{1})\dotsb (\exists x_{n})\left( {}\right.
\bigwedge\limits_{i}card([x_{i}])=k\ \&\ \bigwedge\limits_{i\neq j}\lnot
(x_{i}Ex_{j})\left. {}\right) .
\end{equation*}
\end{proof}

We will say that a subset $K$ of $\omega \times \omega $ is a 
\emph{character} if there is some equivalence structure with character $K$. This is the
same as saying that for all $n$ and $k$, if 
\begin{equation*}
\langle k,n+1\rangle \in K\Rightarrow \langle k,n\rangle \in K.
\end{equation*}

\begin{lemma}
\label{l2a} For any $\Sigma _{2}^{0}$ character $K$, there is a computable
equivalence structure $\mathcal{A}$ with character $K$, which has infinitely
many infinite equivalence classes. Furthermore, in this structure, $\{a:[a]\ 
\text{is finite}\}$ is a $\Pi _{1}^{0}$ set.
\end{lemma}

\begin{proof}
Let $R$ be a computable relation such that 
\begin{equation*}
\langle k,n \rangle \in K \iff (\exists w)(\forall z) R(k,n,w,z).
\end{equation*}

\noindent We will define a set $B$ of quadruples $\langle
k,n,w,z\rangle $ such that when we look at numbers only below $z$, we
believe that $w$ is the least witness that $k,n\in K$, and such that
for other initial segments below $z$, there is some $v<w$ that could
be such a witness. Define the set $B$ as follows:
\begin{eqnarray*}
B &=&\{\langle k,n,w,z\rangle :(\forall y<z)(R(k,n,w,y))\ \& \\
(\forall v &<&w)(\exists y<z)(\lnot R(k,n,v,y))\newline
\&\  \\
(\forall y &<&z)(\exists v<w)R(k,n,v,y)\}\text{.}\ 
\end{eqnarray*}
\linebreak 

The set $B$ is a computable subset of $\omega $ with an infinite complement.
The equivalence structure $\mathcal{A}$ will consist of one class for each
element of $B$, together with an infinite family of infinite equivalence
classes. Partition $\omega \setminus B$ into two computable, infinite,
disjoint subsets. Use the first subset to define the infinitely many
infinite classes, and let the second subset be $C=\{c_{0},c_{1},\dots \}$.
The classes with representatives from $B$ are defined in stages. Let 
$B=\{b_{0},b_{1},\dots \}$ and let $b_{i}=\langle
k_{i},n_{i},w_{i},z_{i}\rangle $.

At \emph{stage }$0$, we put $\{c_{0},\dots ,c_{k_{0}-2}\}$ into the
equivalence class of $b_{0}$.

After stage $s$, we have $s$ equivalence classes with representatives 
$b_0,\dots,b_{s-1}$. For some $i < s$, the classes with representatives $b_i$
have size $k_i$, and others have been declared to be infinite and have at
least $s$ elements. These partial classes contain elements $%
c_0,\dots,c_{p(s)}$ from $C$.

At \emph{stage}\textit{\ }$s+1$, we check for all $i\leq s$ and all
$z\leq s$ whether $R(k_{i},n_{i},w_{i},z)$. If the class $[b_{i}]$ has
previously been declared to be infinite, we simply add one new element
to this class. For any other $i$ such that some
$R(k_{i},n_{i},w_{i},z)$ fails with $z<s$, we declare $[b_{i}]$ to be
infinite and add $s$ new elements from $C$ to $[b_{i}]$. If $b_{s+1}$
is not declared to be infinite, then we put $k_{s+1}-1$ elements from
$C$ into this class.

Let $\mathcal{A}$ be the structure constructed by this process. It is
clear that whenever $\langle k,n\rangle \in K$, then there will be a
unique $b_{i}=\langle k,n,w_{i},z_{i}\rangle $ such that $[b_{i}]$ has
size $k$ and that these are the only finite classes of
$\mathcal{A}$. Thus, $\chi (\mathcal{A})=K$. If $\langle k,n\rangle
\notin K$, then the class $[b_{i}]$ corresponding to $\langle
k,n\rangle $ will be infinite and we already have infinitely many
infinite equivalence classes, so, clearly, $\mathcal{A}$ has
infinitely many infinite equivalence classes.

We note that for any pair $(k,n)$, there is at most one $w$ such that for
all $z$, we have $(k,n,w,z)\in R$ and such that for some $\tilde{z}$, we
have $(k,n,w,\tilde{z})\in B$. Also, for each triple $(k,n,w)$, there is at
most one $z$ such that $(k,n,w,z)\in B$, and for each triple $(k,n,z)$ there
is at most one $w$ such that $(k,n,w,z)\in B$. Now it is clear that $[b_{i}]$
is finite if and only if $(\forall z)R(k_{i},n_{i},w_{i},z),$ and $[b_{i}]$
has $k_{i}$ elements. Further, if $k,n\in K$, then 
\begin{equation*}
|\{[b_{i}]:|[b_{i}]|=k\}|=n\text{.}
\end{equation*}
Then for $c\notin B$, $[c]$ is finite if and only if 
\begin{equation*}
(\forall i)(cEb_{i}\Rightarrow \lbrack b_{i}]\ \text{is finite})\text{.}
\end{equation*}
This completes the proof.
\end{proof}

\begin{lemma}
\label{l2b} For any $r\leq \omega $ and any bounded character $K$
(whether $K\in \Sigma _{2}^{0}$ or not), there is a computable equivalence structure 
$\A$ with character $K$, which has exactly $r$ infinite equivalence
classes. Furthermore, $\{a:[a]^{\mathcal{A}}\ \text{is finite}\}$ is a
computable set.
\end{lemma}

\begin{proof}
The desired structure $\mathcal{A}$ will have three components, each itself
either finite or computable. First, there will be $r$ infinite equivalence
classes. Second, there will be a finite set $\{k_1,\dots,k_m\} \subset
\omega $ and an infinite family of equivalence classes of size $k_i$ for 
$i=1,\dots,m$. Third, there will be a finite set $\{j_1,\dots,j_p\} \subset
\omega$ with corresponding natural numbers $n_i > 0$ for $i=1,\dots,p$ and 
$j_i$ equivalence classes of size $n_i$. It is clear that a computable
structure $\mathcal{A}$ can be constructed such that each desired component
is itself computable.
\end{proof}

The proof of Lemma \ref{l2a} really needed the assumption of infinitely many
infinite equivalence classes, since it is possible that either finitely many
or infinitely many infinite equivalence classes come from $B$.

If there are just finitely many infinite equivalence classes, then the
notion of $s$-functions and $s_{1}$-functions is important. These functions
were introduced by Khisamiev in \cite{khis92}.

\begin{defn}
The function $f:\omega ^{2}\rightarrow \omega $ is an $s$\emph{-function} if
the following hold:

\begin{enumerate}
\item[(i)] For every $i$ and $s$, $f(i,s)\leq f(i,s+1)$.

\item[(ii)] For every $i$, the limit $lim_{s}f(i,s)$ exists.

Let $m_{i}=_{def}lim_{s}f(i,s)$.

We say that $f$ is an \emph{$s_{1}$-function} if, in addition, $m_{i}<m_{i+1}
$ for each $i$.
\end{enumerate}
\end{defn}

\begin{lemma}
\label{l3} Let $\mathcal{A}$ be a computable equivalence structure with
finitely many infinite equivalence classes and an infinite character. Then

\begin{enumerate}
\item[(i)] There exists a computable $s$-function $f$ with
corresponding limits $m_{i}=lim_{s}f(i,s)$ such that $\langle
k,n\rangle \in \chi(\mathcal{A})$ if and only if
\begin{equation*}
card(\{i:k=m_{i}\})\geq n\text{.}
\end{equation*}

\item[(ii)] If the character is unbounded, then there is a computable $s_{1}$
-function $g$ such that $\mathcal{A}$ contains an equivalence class of size 
$m_{i}$ for all $i$, where 
\begin{equation*}
m_{i}=lim_{s}g(i,s).
\end{equation*}
\end{enumerate}
\end{lemma}

\begin{proof}
We may assume, without loss of generality, that $\mathcal{A}$ has no
infinite equivalence classes, since the infinite classes can be captured by
a finite set of representatives.

(i) Define a computable sequence of representatives for all equivalence
classes of $\mathcal{A}$ by setting $a_{0}=0$ and setting $a_{i+1}$ to be
the least $a>a_{i}$ such that $\lnot (aEa_{j})$ for all $j\leq i$. Now
simply let 
\begin{equation*}
f(i,s)=card(\{a\leq s:aEa_{i}\}\text{.}
\end{equation*}

(ii) We will define a uniformly computable family $a_{i}^{s}$ for $i\leq s$
in such a way that $a_{i}=lim_{s}a_{i}^{s}$ converges. We will also define a
computable sequence $p_{s}$, and let 
\begin{equation*}
f(i,s)=card(\{a\leq p_{s}:aEa_{i}^{s}\})\text{.}
\end{equation*}
Hence, we will have 
\begin{equation*}
m_{i}=lim_{s}\;(card(\{a\leq s:aEa_{i}\})=card([a_{i}]))\text{.}
\end{equation*}%
At stage $0$, we have $p_{0}=0$ and $a_{0}^{0}=0$, so that $f(0,0)=1$.

After stage $s$, we have $p_{s}$ and $a_{0}^{s},\dots ,a_{s}^{s}$ such that 
\begin{equation*}
f(i,s)=card(\{a\leq p_{s}:aEa_{i}^{s}\})\text{,}
\end{equation*}
and 
\begin{equation*}
f(0,s)<f(1,s)<\dotsb <f(s,s).
\end{equation*}
At stage $s+1$, we look for the least $p$ such that there is a sequence
$b_{0},\dots ,b_{s+1}$ with the property that 
\begin{equation*}
k_{i}=card(\{a\leq p:aEb_{i}\})
\end{equation*}
and $k_{0}<k_{1}<\dots <k_{s+1}$, and with the further property that
$b_{i}=a_{i}^{s}$ whenever there is no $j\leq i$ and no $aEa_{j}^{s}$
with $p_{s}<a\leq p$. Then we let $a_{i}^{s+1}=b_{i}$ and
$p_{s+1}=p$. To see that such $p$ exists, simply let $m$ be the
largest such that $[a_{j}^{s}]=\{a\leq p:aEa_{j}^{s}\}$ for all
$j\leq m$, and let $b_{i}=a_{i}^{s}$ for all $i\leq m$. Then use the
fact that $\chi (\A)$ is unbounded to find $b_{m+1},\dots ,b_{s+1}$
with
\begin{eqnarray*}
card([a_{m}^{s}]) &<&card([b_{m+1}])<card([b_{m+2}])<\dotsb  \\
&<&card([b_{s+1}])\text{,}
\end{eqnarray*}

\noindent and take $p$ large enough so that $[b_{i}]=\{a\leq p:aEb_{i}\}$.
\end{proof}

\begin{lemma}
\label{nl1} For any computable $s_1$-function $f$, the range of $f$ is a
$\Delta^0_2$ set.
\end{lemma}

\begin{proof}
Let $m_{i}=lim_{s}f(i,s)$. Since $m_{0}<m_{1}<\dots $, it follows that $m\in
ran(f)$ if and only if there exists $i<m$ such that $m=f(i)$, which has the
following two characterizations: 
\begin{equation*}
(\exists s)(\forall t>s)f(i,t)=m\text{,}
\end{equation*}
and 
\begin{equation*}
(\forall s)(\exists t>s)f(i,t)=m\text{.}
\end{equation*}
Thus, the range of $f$ is both $\Sigma _{2}^{0}$ and $\Pi _{2}^{0}$.
\end{proof}

\begin{lemma}
\label{l4} Let $K$ be a $\Sigma^0_2$ characteristic and let $r$ be finite.

\begin{enumerate}
\item[(i)] Let $f$ be a computable $s$-function with the corresponding
limits $m_{i}=lim_{s}f(i,s)$ such that $\langle n,k\rangle \in K$ if and
only if 
\begin{equation*}
card(\{i:k=m_{i}\})\geq n\text{.}
\end{equation*}
Then there is a computable equivalence structure $\mathcal{A}$ with
$\chi(\A)=K$ and with exactly $r$ infinite equivalence classes.

\item[(ii)] Let $f$ be a computable $s_{1}$-function with corresponding
limits $m_{i}=lim_{s}f(i,s)$ such that $\langle m_{i},1\rangle \in K$ for
all $i$. Then there is a computable equivalence structure $\mathcal{A}$ with 
$\chi (\mathcal{A})=K$ and exactly $r$ infinite equivalence classes.
\end{enumerate}
\end{lemma}

\begin{proof}
Clearly, it suffices to prove the statements for $r=0$. We may assume that
$f(i,0)\geq 1$ for all $i$.

(i) Let $a_{i}=2i$. We will build an equivalence structure $\mathcal{A}$
with equivalence classes $\left[ a_{i}\right] $ of $m_{i}$. At stage $0$,
make the elements $1,3,\ldots ,2(f\left( 0,0\right) -1)$ equivalent to $a_{0}
$. After stage $s$, we have exactly $f\left( m,s\right) $ elements
equivalent to $a_{m}$ for each $m\leq s$. Then we add $f(m,s+1)-f(m,s)$
elements to $\left[ a_{m}\right] $ for $m\leq s$ and put $f(s+1,s+1)-1$
elements into the class of $a_{s+1}$.

(ii) Since there is an $s_1$-function, the character $K$ must be unbounded.
We modify the argument for Lemma \ref{l2a} as follows. The pool of elements
to put into the equivalence classes is now simply $\omega \setminus B$.

Here is the first modification. When we find $\neg R(k_i,n_i,w_i,z)$ for
some $z$ at stage $s+1$, we can no longer create an infinite equivalence
class, but we have already put $k_i$ elements in the equivalence class of 
$b_i$. So we will set this bloc $[b_i]$ aside until we find a number $j$ and
a stage $s$ such that $k_i \leq f(j,s)$. Since there is an increasing
sequence $m_0 < m_1 < \dotsb$ corresponding to the $s_1$-function, such $j$
and $s$ will eventually be found. Then we will assign a marker $j$ to $b_i$
and add $f(j,s) - k_i$ elements to the bloc to create an equivalence class
with $f(j,s)$ elements. If at a later stage $t$ we have $f(j,t) > f(j,s)$,
then we will add $f(j,t)-f(j,s)$ more elements to the class. Since $lim_s
f(j,s) = m_j$ converges, we will eventually have an equivalence class of
size $m_j$.

This means that we may have created an extra equivalence class with $f(j,s)$
elements, so the second modification is that when we create a class with 
$k=f(j,s)$ elements, we may need (perhaps temporarily) to remove from our
construction any class corresponding to $\langle k,1,w,z\rangle $. That is,
we set these (finitely many) blocs aside to be put into a larger class, just
as if we had found that $\lnot R(k,n,w,z)$, but we make a note that they may
need to be revived later. If at some later stage $t$, we find $k^{\prime }$
such that 
\begin{equation*}
k^{\prime }=f(j,t)>f(j,s)=k\text{,}
\end{equation*}
so that we will increase the size of the class with marker $j$, then we are
going to remove the classes corresponding to $\langle k^{\prime },1\rangle $
and at the same time revive the classes corresponding to $\langle k,1\rangle 
$. At this stage, we remove the attachment to $f(j,s)$ of the bloc and check
to see for all $b_{i}=\langle k,1,w,z\rangle $:

\begin{tabular}{rl}
(1) & Whether $R(k,1,w,z)$ still holds for all $z\leq t$; \\ 
(2) & Whether the bloc corresponding to $b_{i}$ has been put into a \\ 
& larger class yet.
\end{tabular}

\noindent If (1) is false and (2) is true, then there is nothing else to do.
If (1) and (2) are both false, then we keep the bloc aside for later use. If
(1) is true and (2) is false, then we revive this bloc. If both (1) and (2)
are true, then we create a new class with $k$ elements for each $\langle
k,1,w,z \rangle$ such that $R(k,1,w,z)$ still holds for all $z \leq t$.

We will now describe the construction in detail. Set $C_{-1}=\omega -B$. No 
$j$-markers are used at stage $-1$. If $b_{i}=\langle
k_{i},n_{i},w_{i},z_{i}\rangle $ as in Lemma \ref{l2a}, let us say that 
$b_{i}$ is \emph{active} at stage $s$ if for all $z\leq s$, $
R(k_{i},n_{i},w_{i},z)$, and otherwise we say that $b_{i}$ is \emph{inactive}.

After $s$ stages, we will have some equivalence classes with active
representatives $b_{i}=\langle k,n,w,z\rangle $ or revived representatives 
$b_{i}^{\prime }$, containing $k$ elements, such that for all $z\leq s$, 
$R(k,n,w,z)$. We will have some equivalence classes containing $f(j,s)$
elements corresponding to the $s_{1}$-function $f$. There will also be
certain blocs of size $k_{i}$ corresponding to inactive $b_{i}$ and certain
displaced blocs corresponding to active $b_{i}$, which are waiting to be put
into a larger equivalence class. Finally, there are some active 
$b_{i}=\langle k_{1},1,w,z\rangle $ that have been displaced by some
equivalence class of size $f(j,s)=k_{i}$ but may need to be revived. There
is a pool $C_{s}$ of remaining elements that may be used to fill out new
equivalence classes. At stage $s+1$, the following things may happen.

First, we check to see whether $b_{s+1}=\langle k,n,w,z\rangle $ is active
at stage $s+1$. If so, then we check to see whether $n=1$ and $k=f(j,s+1)$
for some current equivalence class with marker $j$. If such $j$ exists, then
we put $b_{s+1}$ into the pool $C_{s+1}$. Otherwise, we create an
equivalence class with $k$ elements consisting of $b_{s+1}$ and $k-1$
elements from the pool $C_{s}$. If $b_{s+1}$ is already inactive, then we
simply add it to the pool $C_{s+1}$.

Second, we check for $i\leq s$ whether some $b_{i}$ that was active at stage 
$s$ becomes inactive at stage $s+1$. If such $b_{i}$ was representing an
equivalence class at stage $s$, then that class is set aside as a bloc to be
attached to some $f(j,t)$ at a later stage.

Third, we look for the smallest bloc that has been set aside for some
inactive $b_{i}$ at stage $s$, and check whether there exists a previously
unused $j\leq s+1$ such that $k_{i}\leq f(j,s+1)$. If so, then we create an
equivalence class including this bloc and containing $f(j,s+1)$ elements.

Fourth, for all markers $j$ that are being used at stage $s$, we check to
see whether $f(j,s+1)>f(j,s)$. If so, then we add $f(j,s+1)-f(j,s)$ elements
to the corresponding equivalence class. We then displace any classes with
representatives $b_{i}$ for $i\leq s$ such that $k_{i}=f(j,s+1)$ and $n_{i}=1
$. That is, we set aside this class as a bloc to be attached later to some 
$f(j^{\prime },t)$. Finally, we revive any active $b_{i}$ such that 
$k_{i}=f(j,s)$ and $n_{i}=1$, which were displaced by $f(j,s)$. This means
that we create a completely new class with $k_{i}$ elements and a new
representative $b_{i}^{\prime }$ taken from the pool.

It is clear that eventually all elements from the pool are put into some
equivalence class. It needs to be verified that this class eventually
stabilizes at some finite size $k$ and that the resulting equivalence
structure has the desired character $\chi(\mathcal{A})$.

Suppose that $a$ is first put into some class attached to marker $j$ at
stage $s$. Then this class will have size $f(j,t)$ at any later stage $t$
and will stabilize with $m_{j}$ elements. Next, suppose that $a$ is first
put into some bloc with representative $b_{i}$ or $b_{i}^{\prime }$. There
are two cases. If $b_{i}$ remains active at all stages and is never
displaced by any $f(j,t)$, then this class has exactly $k_{i}$ elements at
all future stages. Otherwise, this class is set aside as a bloc at some
later stage, and then eventually put into a class with some marker $j$,
which will stabilize with $m_{j}$ elements. This is guaranteed by the fact
that there are infinitely many $m_{j}>k_{i}$ and eventually the bloc
containing $b_{i}$ will have the highest priority. Thus, all equivalence
classes stabilize at some finite size. Hence, $\mathcal{A}$ has no infinite
equivalence classes.

Now, fix a finite $k$ and suppose that $\langle k,n\rangle \in K$ for all 
$n<r$, where $r\leq \infty $. We need to verify that there are exactly $r$
classes of size $k$ in $\mathcal{A}$. For each $n<r$, there will be a unique
representative $b_{i}=\langle n,k,w,z\rangle $ that remains active at all
stages, where $w$ is the least such that $(\forall z)R(n,k,w,z)$. For $n>1$,
this $b_{i}$ will represent a class of size $k$ and can never be displaced.
For $n=1$, there are two possibilities. There can be some (unique) marker $j$
such that $m_{j}=k$, which corresponds to a class stabilizing at size $k$.
In this case, any classes corresponding to $b_{i}$ (or any later 
$b_{i}^{\prime }$) will be displaced and eventually not revived (once $f(j,s)$
converges to $m_{j}$). On the other hand, if there is no such marker $j$,
then eventually there will be a unique class with representative $b_{i}$ (or
some $b_{i}^{\prime })$ with $k$ elements, which is never displaced. Classes
represented by other $b_{p}$ or by other markers can never have size $k$.
Thus $\chi (\mathcal{A})=K$.
\end{proof}

\bigskip 

The necessity of the $s_{1}$ function is seen by the following.

\begin{theorem}
\label{t3} There is an infinite $\Delta _{2}^{0}$ set $D$ such that for any
computable equivalence structure $\mathcal{A}${\ }with unbounded
character $K$ and no infinite equivalence classes, $\{k:\langle
k,1\rangle \in K\}$ is not a subset of $D$. Hence, for any $s_{1}$
function $f$ with $m_{0}<m_{1}<\dotsb$, where $m_{n}=lim_{s}f(n,s)$,
there exists $i$ such that $m_{i}\notin D$.
\end{theorem}

\begin{proof}
We use a method similar to Post's construction of a simple set. Let 
$\A_{e}$ be the structure with universe $\omega $ and relation 
$E_{e}$ defined by 
\begin{equation*}
iE_{e}j\iff \langle i,j\rangle \in W_{e}
\end{equation*}
and let $[a]_{e}=\{j:aE_{e}j\}$. Then every computable equivalence structure
is $\mathcal{A}_{e}$ for some $e$, and $[a]_{e}$ will be the equivalence
class of $a$. Define a c.e. relation $R$ by 
\begin{equation*}
R(e,a)\iff card([a]_{e})>2e.
\end{equation*}
Then by the standard uniformization theorem for c.e.\ relations (see Soare 
\cite{Soa87}, p.\ 29), there exists a partial computable function $f$,
called a \emph{selector for }$R$, such that, for every $e$, 
\begin{equation*}
(\exists a)R(e,a)\Rightarrow R(e,f(e))\text{.}
\end{equation*}
Define $D$ as follows. 
\begin{equation*}
k\in D\iff (\forall e<\frac{k}{2})(card([f(e)]_{e})\neq k)\text{.}
\end{equation*}
Then $D$ is a $\Delta _{2}^{0}$ set by part (\textit{a)} of Lemma \ref{l1}.
For any $\ell $, the set 
\begin{equation*}
\hat{D}=\{n|(\exists x<\ell )(n=card([f(x)]_{x})\}
\end{equation*}
has cardinality at most $\ell $, so that at most $\ell $ of the elements
from the set $\{0,1,\dots ,2\ell \}$ may be in $\hat{D}$. Thus, the
complement of $D$ contains at most $e$ elements from $\{0,1,\dots ,2e\}$.
Hence, the complement of $D$ is infinite. Now suppose that $\mathcal{A}$ has
unbounded character and has no infinite equivalence classes. Choose $e$ so
that ${\mathcal{A}}={\mathcal{A}}_{e}$. Since $\chi (\mathcal{A})$ is
unbounded, there exists $a$ such that $R(e,a)$, so that $a=f(e)$. Since 
$\mathcal{A}$ has no infinite equivalence classes, 
\begin{equation*}
card([a]^{\mathcal{A}})=card([f(e)]_{e})=k>2e
\end{equation*}
is finite. Then by definition, $\langle k,1\rangle \in \chi (\mathcal{A})$
but $k\notin D$.

Now let $f$ be any $s_{1}$-function, let $m_{i}=lim_{s}f(i,s)$, and
let $K=\{\langle m_{i},1\rangle :i\in \omega \}$. Then there is an
equivalence structure $\mathcal{A}$ with character $K$. Therefore,
some $m_{i}\notin D$.
\end{proof}

\subsection{Computable categoricity of equivalence structures}

We first investigate relative computable categoricity of computable
equivalence structures by showing that they have a formally c.e. Scott
family. 

\begin{proposition}
\label{p1} If $\mathcal{A}$ is a computable equivalence structure with only
finitely many finite equivalence classes, then $\mathcal{A}$ is relatively
computably categorical.
\end{proposition}

\begin{proof}
Choose parameters $c_{1},\dots ,c_{n}$ which are representatives of
the $n$ finite equivalence classes. A Scott formula for any finite
sequence $\overrightarrow{a}=a_{1},\dots ,a_{m}$ of elements from $A$
is a conjunction of two formulas. The first formula $\phi
(\overrightarrow{x})$ is the conjunction of all formulas
$x_i E x_j$ (when $a_i E^{\mathcal{A}} a_j$) and $\lnot
(x_{i}Ex_{j})$ (when it is not the case that
$a_{i}E^{\mathcal{A}}a_{j}$). The second formula $\psi
(\overrightarrow{x},{\overrightarrow{c}})$ is the conjunction of all
formulas $x_{i}Ec_{j}$ (when $a_{i}E^{\mathcal{A}}c_{j}$) and $\lnot
(x_{i}Ec_{j})$ (when it is not the case that
$a_{i}E^{\mathcal{A}}c_{j}$). It is clear that every tuple of elements
from $A$ satisfies one of these formulas.

Suppose that $\overrightarrow a$ and $\overrightarrow b$ satisfy the same
Scott formula. Then, in particular, we have $a_i E^{\mathcal{A}} a_j \iff
b_i E^{\mathcal{A}} b_j$.

For any tuple ${\overrightarrow d}$, the equivalence class $[d_i]$ is finite
if and only if some $x_i E^{\mathcal{A}} c_j$ occurs in the Scott formula of 
$\overrightarrow d$, and $[d_i]$ is infinite otherwise. We will define an
automorphism $H$ of $\mathcal{A}$ mapping $\overrightarrow a$ to 
$\overrightarrow b$.

For any equivalence class $[a]$ containing none of the elements of
$\overrightarrow a$, of $\overrightarrow b$, or of $\overrightarrow c$, the
function $H$ will simply be the identity map. We also define $H(a_i) = b_i$
and $H(c_i) = c_i$. This induces a partial one-to-one map from the
equivalence classes of $\mathcal{A}$ into the equivalence classes of 
$\mathcal{A}$, which fixes finite classes setwise and takes infinite classes
to infinite classes. Within a particular finite class $[a_i]$ of size $n$,
the partial map from $[a_i]$ to $[b_i]$ defined on $[a_i] \cap
\{\overrightarrow a\}$ can be extended to an isomorphism of $[a_i]$ onto $[b_i]$.

For the infinite classes (whether finitely or infinitely many) the partial
isomorphism of the classes may similarly be extended to a total isomorphism
of the classes. Likewise, the partial map taking $a_i$ to $b_i$ may be
extended to map the infinite class $[a_i]$ to the infinite class $[b_i]$.
\end{proof}

\begin{proposition}
\label{p2} Let $\mathcal{A}$ be a computable equivalence structure with
finitely many infinite classes, with bounded character and with at most one
finite $k$ such that there are infinitely many equivalence classes of size 
$k $. Then $\mathcal{A}$ is relatively computably categorical.
\end{proposition}

\begin{proof}
Let $c_{1},\dots ,c_{n}$ be representatives for the finite classes not
of size $k$ and let $d_{1},\dots ,d_{p}$ be representatives for the
finitely many infinite classes. Then the Scott formula for a finite
sequence $\overrightarrow{a}$ from $\mathcal{A}$ is the conjunction of
three formulas, the first two as in the proof of Proposition \ref{p1},
and the third is the conjunction of all formulas $x_{i}Ed_{j}$ (when
$a_{i}E^{\mathcal{A}}d_{j}$) and $\lnot (x_{i}Ed_{j})$ (when it is not
the case that $a_{i}E^{\mathcal{A}}d_{j}$). Then $[a_{i}]$ is infinite
if and only if $a_{i}E^{\mathcal{A} }d_{j}$ for some $j$, and
$card([a_{i}])=k$ if and only if $\lnot (a_{i}E^{ \mathcal{A}}d_{j})$
for all $j$, and also $\lnot (a_{i}E^{\mathcal{A}}c_{j})$ for all $j$.

Suppose that $\overrightarrow a$ and $\overrightarrow b$ satisfy the same
Scott formula. Then we can define an automorphism of $\mathcal{A}$ extending
the partial function which takes $a_i$ to $b_i$ as in the proof of
Proposition \ref{p1}.
\end{proof}

\begin{corollary}
Let $\mathcal{A}$ be a computable equivalence structure of one of the
following types:

\begin{enumerate}
\item $\mathcal{A}$ has only finitely many finite equivalence classes;

\item $\mathcal{A}$ has finitely many infinite classes, has bounded
character (i.e.\ only finitely many finite \emph{sizes} of equivalence
classes), and has at most one finite $k$ such that there are infinitely many
classes of size $k$.
\end{enumerate}

Then $\mathcal{A}$ is relatively computably categorical.
\end{corollary}

For structures $\A$ with $Fin^{\A}$ computable, there is a stronger
result.

\begin{proposition} \label{np1} 
Let $\A$ and $\B$ be isomorphic computable equivalence structures
  such that $Fin^{\A}$ and $Fin^{\B}$ are computable and such that
  $\A$ has infinitely many equivalence classes of size $k$ for at most
  one finite $k$. Then $\A$ and $\B$ are computably isomorphic.
\end{proposition}
 
\begin{proof} By Proposition \ref{p1}, $Inf^{\A}$ and $Inf^{\B}$ are
  computably isomorphic and by Proposition \ref{p2}, $Fin^{\A}$ and
  $Fin^{\B}$ are computably isomorphic. 
\end{proof}

In the remainder of this section, we will show that no other equivalence
structures are computably categorical. Here is the first case.

\begin{theorem}
\label{t4} Suppose that there exist $k_1 < k_2 \leq \omega$ such that
the computable equivalence structure $\mathcal{A}$ has infinitely many
equivalence classes of size $k_1$ and infinitely many classes of size
$k_2$. Then $\mathcal{A}$ is not computably categorical.
\end{theorem}

\begin{proof}
We will define structures $\mathcal{C}$ and $\mathcal{D}$ both isomorphic to 
$\mathcal{A}$ and such that $\{a:card([a]^{\mathcal{C}})=k_{1}\}$ is a
computable set, but 
\begin{equation*}
\{a:card([a]^{\mathcal{D}})=k_{1}\}
\end{equation*}
is a not computable. Then these two structures are not computably
isomorphic, so $\mathcal{A}$ is not computably categorical.

Observe that $\chi (\mathcal{A})$ is a $\Sigma _{2}^{0}$ set by part (c) of
Lemma \ref{l1} and therefore 
\begin{equation*}
K=\chi (\mathcal{A})\setminus \{\langle k_{1},n\rangle :n<\omega \}
\end{equation*}
is also $\Sigma _{2}^{0}$. Thus, if $\chi (\mathcal{A})$ is bounded,
then there is a computable equivalence structure with character $K$
and the same number of infinite equivalence classes as $\A$. If $\chi
(\mathcal{A})$ is unbounded, then by Lemma \ref{l3}, there is an
$s_{1}$-function $f$ for $\chi (\mathcal{A})$. If $k_{1}\neq m_{i}$
for any $i$, then $f$ will be an $s_{1}$-function for the character
$K$. If $k_{1}=m_{i}$, then define a new $s_{1}$-function $g$ for
$K$ by $g(j)=f(j)$ for $i<j$, and $g(j)=f(j+1)$ for $i\geq j$. Then
there is a computable structure $\B$ with character $K$ by
Lemma \ref{l4}. Now we can define a structure $\C\simeq
\mathcal{A}$ (that is, with character $\chi (\mathcal{A})$) by
setting
\begin{equation*}
(2a+1)\;E^{\mathcal{C}}\;(2b+1)\iff aE^{\mathcal{B}}b\text{,}
\end{equation*}
and 
\begin{equation*}
(2(mk_{1}+i))\;E^{\mathcal{C}}\;(2(nk_{1}+j))\iff m=n\text{,}
\end{equation*}
where $i,j<k_{1}$. In this structure $\mathcal{C}$, $\{a:card([a])=k_{1}\}$
is a computable set.

At the same time, we can build a structure $\mathcal{D} \simeq \mathcal{A}$
such that $\{a: card([a]^{\mathcal{D}})=k_1\}$ is not computable. There are
two cases, depending on whether $k_2$ is finite.

First suppose that $k_{2}$ is finite and build a computable structure
$\mathcal{C}^{\prime }$ with character $\{\langle k_{1},n\rangle
,\langle k_{2},n\rangle :n<\omega \}$ in which $\{a:card([a])=k_{1}\}$
is a complete c.e. set, as follows. Let $M$ be a complete c.e. set and
note that $M$ is both infinite and co-infinite. The equivalence
classes of $\mathcal{C}^{\prime }$ will have representatives $2i$ so
that $card([2i])=k_{1}$ if $ i\notin M$, and $card([2i])=k_{2}$ if
$i\in M$. The odd numbers will act as a pool of elements to fill out
the classes.

The construction of $\mathcal{C}^{\prime }$ is in stages. After stage $s$,
there will be classes $C_{i}^{s}$ containing $2i$ which will have $k_{1}$
elements if $i\notin M_{s}$ and $k_{2}$ elements if $i\in M_{s}$. At stage 
$s+1$, we add a new class containing $2s+2$ and also containing $k_{1}-1$ new
odd elements from the pool if $s+1\notin M_{s+1}$ and containing $k_{2}-1$
new elements from the pool if $s+1\in M_{s+1}$. Also, for any $i\leq s$ such
that $i\in M_{s+1}\setminus M_{s}$, we will add $k_{2}-k_{1}$ new elements
from the pool to the class with representative $[2i]$.

Next suppose that $k_2 = \omega$. Just modify the construction above so that
when $i \in M_s$, the class $[2i]$ contains $max\{k_1,s\}$ elements. The
details are left to the reader.

Finally we combine $\mathcal{A}$ and $\mathcal{C}^{\prime }$ into
$\mathcal{D2}$ by coding $\mathcal{A}$ on the odd numbers and
$\mathcal{C}^{\prime }$ on the even numbers. Then
\begin{equation*}
card([2a])=k_{1}\iff a\notin M\text{,}
\end{equation*}
so that $\{d:card([d])=k_{1}\}$ is not computable.
\end{proof}

We observe that in the proof of Theorem \ref{t4}, if $Fin^{\A}$ is
computable, then for the case that $k_2 < \omega$, $Fin^{\C}$ and
$Fin^{\D}$ will also be computable. Thus we have the following. 

\begin{proposition} \label{np2} For any $k_1 < k_2 < \omega$ and any 
computable equivalence structure $\A$ with $Fin^{\A}$ computable and
with infinitely many equivalence classes of size $k_1$ and infinitely
many equivalence classes of size $k_2$, there is a computable
equivalence structure $\B$ isomorphic to $\A$ with $Fin^{\B}$
computable such that $\A$ and $\B$ are not computably isomorphic.  
\end{proposition}

\bigskip 

For the next result, we want to consider the so-called isomorphism problem
for a class of structures. For total recursive functions $\phi _{e}:\omega
\times \omega \rightarrow \{0,1\}$, let $\mathcal{C}_{e}=(\omega ,\equiv
_{e})$ be the structure with 
\begin{equation*}
m\equiv _{e}n\iff \phi _{e}(\langle m,n\rangle )=1.
\end{equation*}
It is easy to check that $\{e:\mathcal{C}_{e}\ \text{is an equivalence
structure}\}$ is a $\Pi _{2}^{0}$ set. The following Lemma is immediate from Lemma %
\ref{l1}. \medskip 

\textbf{Convention:} We say that a set is $D_{3}^{0}$ if it is a difference
of two $\Sigma _{3}^{0}$ sets.

\begin{lemma}
\label{l4a}

\begin{enumerate}
\item[(a)] For any finite $r$, \newline $\{e: C_e \ \text{has at least
$r$ infinite equivalence classes}\}$ is a $\Sigma^0_3$ set.

\item[(b)] For any finite $r$, $\{e: C_e \ \text{has at exactly $r$ infinite
equivalence classes}\}$ is a $D^0_3$ set.

\item[(c)] $\{e: C_e \ \text{has infinitely many infinite equivalence
classes}\}$ is a $\Pi^0_4$ set.
\end{enumerate}
\end{lemma}

We need to look at indices for $\Sigma _{2}^{0}$ sets in general. Let
$\langle S_{e}:e\in \omega \rangle $ be an enumeration of the $\Sigma
_{2}^{0} $ sets, that is,
\begin{equation*}
n\in S_{e}\iff (\exists m)(\langle m,n\rangle \notin W_{e})\text{.}
\end{equation*}
Then an enumeration $\langle K_{e}:e\in \omega \rangle $ of the $\Sigma
_{2}^{0}$ characters may be defined by 
\begin{equation*}
\langle k,n\rangle \in K_{e}\iff (\forall j\leq n)(\langle k,j\rangle \in
S_{e})\text{.}
\end{equation*}

\begin{lemma}
\label{l5} Let $K$ be a fixed infinite $\Sigma^0_2$ character. Then $\{e:
K_e = K\}$ is $\Pi^0_3$ complete.
\end{lemma}

\begin{proof}
The set $\{e:K_{e}=K\}$ is certainly a $\Pi _{3}^{0}$ set. Now, let $P$ be
an arbitrary $\Pi _{3}^{0}$ set and let $S$ be a $\Sigma _{2}^{0}$ set such
that 
\begin{equation*}
e\in P\iff (\forall k)(\langle k,e\rangle \in S)\text{,}
\end{equation*}%
where we may assume, without loss of generality, that 
\begin{equation*}
\langle k+1,e\rangle \in S\Rightarrow \langle k,e\rangle \in S\text{.}
\end{equation*}
Next, we define a computable function $f$ such that 
\begin{equation*}
K_{f(e)}=K\iff e\in P\text{.}
\end{equation*}
Set 
\begin{equation*}
\langle k,n\rangle \in K_{f(e)}\iff \langle k,n\rangle \in K\ \&\ \langle
k,e\rangle \in S\ \&\ \langle n,e\rangle \in S\text{.}
\end{equation*}%
If $e\in P$, then $\langle k,e\rangle \in S$ for all $k$, so that for every 
$k$ and $n$, 
\begin{equation*}
\langle k,n\rangle \in K_{f(e)}\iff \langle k,n\rangle \in K\text{.}
\end{equation*}
If $e\notin P$, then there is some $k_{0}$ such that, for all $k\geq k_{0}$, 
$\lnot S(k,e)$. Since $K$ is infinite, there is some $\langle k,n\rangle \in
K$ such that either $k\geq k_{0}$ or $n\geq k_{0}$, and, therefore, $\langle
k,n\rangle \notin K_{f(e)}$.
\end{proof}

\bigskip 

We note that in the proof of Lemma \ref{l5}, $K_{f(e)}\subseteq K$ for
all $e $.


\begin{theorem}
\label{isom1} Let $\mathcal{A}$ be a  computable equivalence structure 
with character $K$ such that there does not exist a computable equivalence structure
$\mathcal{B}$ with character $K$ and with finitely many infinite
equivalence classes. Then $\{e: \mathcal{C}_e \simeq \mathcal{A} \}$
is $\Pi^0_3$ complete.
\end{theorem}

\begin{proof}
The set $\{e:\mathcal{C}_{e}\simeq \mathcal{A}\}$ is $\Pi _{3}^{0}$ by
Lemmas \ref{l1} and \ref{l5}, since $\mathcal{C}_{e}\simeq \mathcal{A}$ if
and only if $\chi (\mathcal{C}_{e})=K$.

For the completeness, let the computable function $f$ be as in the proof of
Lemma \ref{l5}. Use the technique of Lemma \ref{l2a} uniformly to create the
equivalence structure $\mathcal{C}_{g(e)}$ with character $K_{f(e)}$ and
infinitely many infinite equivalence classes. Then $\mathcal{C}_{g(e)}$ is
isomorphic to $K$ if and only if $K_{f(e)}=K$. The result now follows by
Lemma \ref{l5}.
\end{proof}

\bigskip 

\begin{theorem}
\label{u1} Let $K$ be an unbounded $\Sigma _{2}^{0}$ character. Let $%
\mathcal{A}$ be a computable equivalence structure with character $K$ such that there
does not exist a comptable equivalence structure $\mathcal{B}$ with character $K$ and
with finitely many infinite equivalence classes. Then $\mathcal{A}$ is not
computably categorical.
\end{theorem}

\begin{proof}
If $\mathcal{A}=(\omega ,$ $\equiv _{A})$ is computably categorical, then $%
\mathcal{C}_{e}\simeq \mathcal{A}$ if and only if $\mathcal{A}$ and $%
\mathcal{C}_{e}$ are computably isomorphic. But this has a $\Sigma _{3}^{0}$
definition, that is, 
\begin{equation*}
(\exists a)[a\in Tot\ \&\ (\forall m)(\forall n)(m\equiv _{e}n\iff \phi
_{a}(m)\equiv _{A}\phi _{a}(n))]\text{.}
\end{equation*}%
This contradicts the $\Pi _{3}^{0}$ completeness from Theorem \ref{isom1}.
\end{proof}

\bigskip 

For characters with $s_{1}$ functions, a structure may have finitely many or
infinitely many infinite equivalence classes, and there is a higher
complexity.

\begin{theorem}
\label{isom2} Let $\mathcal{A}$ be a computable equivalence structure with unbounded
character $K$ and with a finite number $r$ of infinite equivalence classes.

\begin{enumerate}
\item[(a)] If $r=0$, then $\{e:C_{e}\simeq \mathcal{A}\}$ is $\Pi _{3}^{0}$
complete.

\item[(b)] If $r>0$, then $\{e:C_{e}\simeq \mathcal{A}\}$ is $D_{3}^{0}$
complete.
\end{enumerate}
\end{theorem}

\begin{proof}
(a) Suppose that $\mathcal{A}$ has no infinite equivalence classes. Then 
$\{e:C_{e}\simeq \mathcal{A}\}$ is a $\Pi _{3}^{0}$ set, since $C_{e}\simeq 
\mathcal{A}{}$ if and only if the following two facts hold:

(1) $\chi (C_{e})=K$ (which is a $\Pi _{3}^{0}$ condition by Lemma
    \ref{l5});
\smallskip \noindent

(2) $C_e$ has no infinite equivalence classes (which is a $\Pi^0_3$
condition by Lemma \ref{l4a}).

For the completeness, let $P$ be a given $\Pi _{3}^{0}$ set. We
construct a reduction of $P$ to our set as follows. Let $g$ be an
$s_{1}$-function for $K $, let $m_{i}=lim_{s}g(i,s)$ and let
$f_{0}(i,s)=g(i,2s)$ and $f_{1}(i,s)=g(i,2s+1)$ so that $f_0$ and
$f_1$ are both $s_1$-functions. Now let
\begin{equation*}
K_{1}=K\setminus \{\langle m_{2i},n\}:i<\omega \}.
\end{equation*}
Let $\phi $ be given by the proof of Lemma \ref{l5} so that $K_{\phi
(e)}=K_{1}$ if and only if $e\in P$ and such that $K_{\phi (e)}\subseteq
K_{1}$ for all $e$. Then 
\begin{equation*}
K_{\phi (e)}\cup (K\cap \{\langle m_{2i},n\rangle :i\in \omega \})
\end{equation*}
always has an $s_{1}$-function, so we can apply the proof of Lemma \ref{l4}
to construct $C_{\psi (e)}$ with character 
\begin{equation*}
K_{\phi (e)}\cup (K\cap \{\langle m_{2i},n\rangle :i\in \omega \})\text{,}
\end{equation*}
which has no infinite equivalence classes. It is now clear that $e\in P$ if
and only if $K_{\phi (e)}=K_{1}$, which is if and only if $C_{\psi
(e)}\simeq \mathcal{A}$.

Note that if we simply apply Lemma \ref{l5} to $K$ itself, we find that 
$K_{\phi(e)}$ is finite whenever $K_{\phi(e)} \neq K$, so that $C_{\phi(e)}$
would also be finite, whereas we are assuming that all computable
equivalence structures $C_i$ have universe $\omega$.

(b) Suppose that $\mathcal{A}$ has exactly $r>0$ infinite equivalence
classes. Then $\{e:C_{e}\simeq \mathcal{A}\}$ is a $D_{3}^{0}$ set, that is,
the difference of two $\Sigma _{3}^{0}$ sets, since $C_{e}\simeq \mathcal{A}$
if and only if both of the following two facts hold:

(1) $\chi (C_{e})=K$ (which is a $\Pi _{3}^{0}$ condition by Lemma
    \ref{l5}); \smallskip \noindent

(2) $C_e$ has exactly $r$ infinite equivalence classes (which is a $D^0_3$
condition by Lemma \ref{l4a}).

For the completeness, let $P$ be a $\Pi _{3}^{0}$ set as in part (a) and let 
$Q$ be a $\Sigma _{3}^{0}$ set. Now let $R$ be a $\Pi _{2}^{0}$ set such
that, for all $d$, 
\begin{equation*}
d\in Q\iff (\exists c)(\langle c,d\rangle \in R)\text{.}
\end{equation*}
Without loss of generality, we may assume that when $d\in Q$, there exists a
unique $c$ such that $\langle c,d\rangle \in R$. It follows from the $\Pi
_{2}^{0}$ completeness of $\{e:W_{e}\ \text{is infinite}\}$ that there is a
computable set $T$ such that 
\begin{equation*}
\langle c,d\rangle \in R\iff (\{t:\langle c,d,t\rangle \in T\}\ \text{is
infinite})\text{.}
\end{equation*}
We will define a computable function $\theta $ so that, for all $d$ and $e$, 
\begin{equation*}
C_{\theta (d,e)}\simeq \mathcal{A}\iff d\in Q\ \&\ e\in P.
\end{equation*}
$C_{\theta (d,e)}$ will be the disjoint union of three components.

The first component will be a structure $\mathcal{B}$ that has no infinite
equivalence classes and has character $K_{\phi (e)}$, where 
\begin{equation*}
e\in P\iff K_{\phi (e)}=K_{1}\text{.}
\end{equation*}
This is constructed as in part (a).

The second component, $\mathcal{C}$, is fixed for all $e$, has no infinite
equivalence classes and has character 
\begin{equation*}
\{\langle m_{2i},n\rangle :\langle m_{2i},n+1\rangle \in K\}.
\end{equation*}
This might be a finite structure.

The third component, $\mathcal{D}$, will have character $\{\langle m_{2i},1
\rangle: i \in \omega\}$ and will have exactly $r$ infinite equivalence
classes if $d \in Q$ and no infinite equivalence classes if $d \notin Q$. It
suffices to give the argument when $r=1$ since then we can always take $r$
copies of $\mathcal{D}$ to get $r$ infinite equivalence classes.

From the $s_{1}$-function $f_{0}$ we create an infinite set of
$s_{1}$-functions $g_{c}$, where
\begin{equation*}
g_{c}(i,s)=f_{0}(2^{c}(2i+1))\text{.}
\end{equation*}
Let 
\begin{equation*}
m_{c,i}=lim_{s}g_{c}(i,s)\text{.}
\end{equation*}
Then $\mathcal{D}$ will be the disjoint union of equivalence
structures $\mathcal{D}_{c}$ having character
\begin{equation*}
K\cap \{\langle m_{c,i},n\rangle :i,n\in \omega \})\text{,}
\end{equation*}
and having exactly one infinite equivalence class if $\langle c,d\rangle \in
Q$ and no infinite equivalence class otherwise. It now suffices to construct 
$\mathcal{D}_{c}$ with universe $\omega $.

Fix $c$ and let $n_{i}=m_{c,i}$. The construction of the equivalence
relation $E$ on $\mathcal{D}_{c}$ is in stages. At stage $s$, there will be
equivalence classes $C_{i}^{s}$ of size $g_{c}(i,s)$ for all $i<s$. There
will be a particular $i=i^{s}$ such that 
\begin{equation*}
\{2t:\langle c,d,t\rangle \in T\}\subseteq C_{i}^{s}\text{.}
\end{equation*}
This class $C_{i}^{s}$ is the \emph{test} class. \emph{Initially}, we have
the empty structure. By stage $s$, all numbers $<s$ will have been assigned
to an equivalence class, and, hence, we will have decided whether $aEb$ for
all $a,b<2s$.

At \emph{stage} $t+1$, let $i=i^{t}$ and check whether $\langle
c,d,t\rangle \in T$. If it is, then we let $i^{t+1}=t$ and we create
the new class $C_{t}^{t+1}$ by adding to $C_{i}^{t}$ the element $2t$,
along with $g(t,t+1)-g(i,t)-1$ new odd numbers. We also create a new
class $C_{i}^{t+1}$ with $g(i,t+1)$ new odd numbers. For all $j$ such
that $j<t$ and $j\neq i$, we add $g_{c}(j,t+1)-g_{c}(j,t)$ odd numbers
to the class $C_{j}^{t}$ to obtain $C_{j}^{t+1}$.

If $\langle c,d,t \rangle \notin T$, then for all $j<t$, we simply add 
$g_c(j,t+1)-g_c(j,t)$ odd numbers to the class $C_j^t$ to obtain $C_j^{t+1}$
and we create the new class $C^t$ with exactly $g_c(t,t+1)$ new odd elements.

There are two possible outcomes of this construction. If $\{t:\langle
c,d,t\rangle \in T\}$ is finite, then after some stage $t$, $i^{t}$ becomes
fixed and, thus, has limit $i$. Then for every $i$, the class $C_{i}=\cup
_{t}C_{i}^{t}$ will have exactly $n_{i}$ elements, and every number will
belong to one of these classes. Thus, $\mathcal{D}_{c}$ has character 
$\{\langle n_{i},1\rangle :i\in \omega \}$ and has no infinite equivalence
classes. If $\{t:\langle c,d,t\rangle \in T\}$ is infinite, then 
$lim_{t}i^{t}=\infty $ and ${}\mathcal{D}_{c}$ has one additional, infinite
equivalence class, the test class, which is $\cup _{t}C_{i^{t}}^{t}$.

It follows that if $d\notin Q$, then each ${}\mathcal{D}_{c}$ has character 
$\{\langle m_{c,i},1\rangle :i\in \omega \}$ and has no infinite equivalence
classes, so that $\mathcal{D}$ has character $\{\langle m_{2i},1\rangle
:i\in \omega \}$ and has no infinite equivalence classes. If $d\in Q$, then
one of the $\mathcal{D}_{c}$ has one infinite equivalence class and the
others have no infinite equivalence classes. Thus, $\mathcal{D}$ has exactly
one infinite equivalence class, as desired.
\end{proof}

\begin{theorem}
\label{u2} Let $K$ be an unbounded $\Sigma _{2}^{0}$ character, and
let $\mathcal{A}$ be a computable equivalence structure with character $K$ and with
finitely many infinite equivalence classes. Then $\mathcal{A}$ is not
computably categorical.
\end{theorem}

\begin{proof}
This follows immediately from Theorem \ref{isom2} as in the proof of Theorem
\ref{u1}
\end{proof}

Note that since there are finitely many infinite equivalence classes,
for the structures $\A$ and $\B$ of Theorem \ref{u2} which are isomorphic
but not computably isomorphic, both $Fin^{\A}$ and $Fin^{\B}$ are
computable. 

\begin{theorem}
\label{isom3} Let $\mathcal{A}$ be an equivalence structure with an
unbounded character $K$, and with infinitely many infinite equivalence
classes. Suppose that there exists an equivalence structure $\mathcal{B}$
with character $K$ and with finitely many infinite equivalence classes. Then 
$\{e:\mathcal{C}_{e}\simeq \mathcal{A}\}$ is $\Pi _{4}^{0}$ complete.
\end{theorem}

\begin{proof}
It follows from Lemma \ref{l3} that $K$ possesses an $s_{1}$-function
$g$.  Let $m_{i}=lim_{s}g(i,s)$. Let $f_{0}(i,s)=g(i,2s)$ and
$f_{1}(i,s)=g(i,2s+1) $, so that both $f_{0}$ and $f_{1}$ are
$s_{1}$-functions. From the $s_{1}$-function $f_{0}$ we create an
infinite set of $s_{1}$-functions $g_{c}(i,s)$ , where
\begin{equation*}
g_{c}(i,s)=f_{0}(2^{c}(2i+1))\text{.}
\end{equation*}
Set $m_{c,i}=lim_{s}g_{c}(i,s)$. Let 
\begin{equation*}
K_{0}=K\setminus \{\langle m_{2i},n\}:i,n<\omega \}\text{,}
\end{equation*}
and for each $c$, let 
\begin{equation*}
K_{c+1}=K\cap \{\langle m_{c,i},n\rangle :i,n\in \omega \}\text{.}
\end{equation*}
Thus, $K$ is the disjoint union of the characters $K_{c}$. Now, $K_{0}$ has 
$s_{1}$-function $f_{1}$, and, therefore, there is a model $\mathcal{A}_{0}$
with character $K_{0}$ and no infinite equivalence classes.

Let $P$ be a $\Pi _{4}^{0}$ set. Let $Q$ be a $\Sigma _{3}^{0}$ relation
such that 
\begin{equation*}
e\in P\iff (\forall c)Q(e,c)\text{.}
\end{equation*}
We may assume that if $e\notin P$, then $\{c:Q(e,c)\}$ is finite.
Uniformizing the proof of Theorem \ref{isom2}, there is a computable binary
function $\phi $ such that $\mathcal{C}_{\phi (e,c)}$ has character $K_{c+1}$
for all $e$ and $c$, and has exactly one infinite equivalence class if 
$Q(e,c)$, and no infinite equivalence class if $\lnot Q(e,c)$.

Now define ${\mathcal{C}}_{\psi (e)}=(\omega ,E)$ as the effective union of
the structures $\mathcal{A}_{0},{}{\mathcal{C}}_{\phi (e,c)}$. That is, let 
$E_{0}$ be the equivalence relation of $\mathcal{A}_{0}$, and let $E_{e,c}$
be the equivalence relation of ${\mathcal{C}}_{\phi (e,c)}$. Let 
\begin{equation*}
E(2a,2b)\iff E_{0}(a,b)\text{,}
\end{equation*}
and for each $c$, let 
\begin{equation*}
E(2^{c}(2a+1),2^{c}(2b+1))\iff E_{e,c}(a,b)\text{;}
\end{equation*}
for any other $i,j$, we let $\lnot E(i,j)$. Then the structure
${\mathcal{C}}=(\omega ,E)$ clearly has character $K=\cup _{c}K_{c}$.

If $e\in P$, then $Q(e,c)$ holds for all $c$, so each ${\mathcal{C}}_{\phi
(e,c)}$ has an infinite equivalence class. Hence, ${}{\mathcal{C}}_{\psi (e)}$
has infinitely many infinite equivalence classes.

If $e\notin P$, then $Q(e,c)$ holds for finitely many $c$, so finitely many 
${\mathcal{C}}_{\phi (e,c)}$ have exactly one infinite equivalence class, and
the others have no infinite equivalence classes. Thus, ${\mathcal{C}}_{\psi
(e)}$ has finitely many infinite equivalence classes.

It follows that ${\mathcal{C}}_{\psi (e)}\simeq {\mathcal{A}}$ if and only
if $e\in P$.
\end{proof}

\bigskip 

\begin{theorem}
\label{u3} Let $\mathcal{A}$ be a computable equivalence structure with unbounded
character $K$ and with infinitely many infinite equivalence classes, such
that there exists a computable equivalence structure $\mathcal{B}$ with character $K$
and with finitely many infinite equivalence classes. Then $\mathcal{A}$ is
not $\Delta _{2}^{0}$ categorical.
\end{theorem}

\begin{proof}
If $\mathcal{A}=(\omega ,\equiv _{A})$ is $\Delta _{2}^{0}$
categorical, then ${\mathcal{C}}_{e}\simeq {\mathcal{A}}$ if and only
if $\mathcal{A}$ and ${\mathcal{C}}_{e}$ are $\Delta _{2}^{0}$
isomorphic. Thus, the set $\{e:{\mathcal{C}}_{e}\simeq
{\mathcal{A\}}}$ has a $\Sigma _{4}^{0}$ definition.  That is, with a
c.e. complete set $M$ as an oracle, we have
\begin{equation*}
(\exists a)\left( [\right. a\in Tot^{M}\&(\forall m)(\forall n)(m\equiv
_{e}n\Longleftrightarrow \phi _{a}^{M}(m)\equiv _{A}\phi _{a}^{M}(n)\left.
)\right) ].
\end{equation*}
This contradicts the $\Pi _{4}^{0}$ completeness from Theorem \ref{isom3}.
\end{proof}

Combining these results, we obtain the following corollary.

\begin{corollary}
No equivalence structure with an unbounded character is computably
categorical. 
\end{corollary}

We can now establish that for computable equivalence structures computable
categoricity and relative computable categoricity coincide. 

\begin{theorem}
\label{ccc} All computably categorical equivalence structures are also
relatively computably categorical.
\end{theorem}

\begin{proof}
Suppose that $\mathcal{A}$ is not relatively computably categorical
and has character $K$. It follows from Propositions \ref{p1} and
\ref{p2} that $\mathcal{A}$ has infinitely many finite equivalence
classes. First, suppose that $K$ is bounded. Then there exists a
finite $k$ such that $\mathcal{A}$ has infinitely many classes of size
$k$. Now it follows from Proposition \ref{p2} that either
$\mathcal{A}$ has infinitely many infinite classes, or there are two
finite numbers $k_{1}$ and $k_{2}$ such that $\mathcal{A}$ has
infinitely many classes of size $k_{1}$ and infinitely many classes of
size $k_{2}$. In either case, Theorem \ref{t4} implies that
$\mathcal{A}$ is not computably categorical.

Now, suppose that $K$ is unbounded. Then there are two possibilities.
Suppose first that $K$ has no $s_{1}$-function. Then, by Theorem
\ref{u1}, $\A$ is not computably categorical. Next, suppose that $K$
has an $s_{1}$-function and that $\mathcal{A}$ has infinitely many
infinite equivalence classes. Then, by Theorem \ref{u3}, $\mathcal{A}$
is not computably categorical. Finally, suppose that $\mathcal{A}$ has
only finitely many infinite equivalence classes. Then $\mathcal{A}$ is
not computably categorical by Theorem \ref{u2}.
\end{proof}

\section{$\Delta _{2}^{0}$ Categoricity of equivalence structures}

Next we continue with the analysis of $\Delta _{2}^{0}$ categoricity.

\begin{theorem}
\label{rdc1} If $\mathcal{A}$ is a computable equivalence structure with
bounded character, then $\mathcal{A}$ is relatively $\Delta^0_2$ categorical.
\end{theorem}

\begin{proof}
Let $k$ be the maximum size of any finite equivalence class. The key
fact here is that $[a]$ is infinite if and only if $[a]$ contains at
least $k+1$ elements, which is a $\Sigma _{2}^{0}$ condition. By Lemma
\ref{l1}, there is a $\Delta _{2}^{0}$ formula which characterizes the
elements $a$ with a finite equivalence class of size $m$. Then a Scott
formula for the tuple $\langle a_{1},\dots ,a_{m}\rangle $ includes a
formula $\psi _{i}(x_{i})$ for each $a_{i}$, giving the cardinality of
$[a_{i}]$, together with formulas $\psi _{i,j}(x_{i},x_{j})$ for each
$i$ and $j$ which express whether $a_{i}Ea_{j}$. It now follows, as in
the proof of Proposition \ref{p1}, that whenever $\overrightarrow{a}$
and $\overrightarrow{b}$ have the same Scott formula, then there is an
automorphism of $\mathcal{A}$ taking $\overrightarrow{a}$ to
$\overrightarrow{b}$. Thus, $\mathcal{A}$ is relatively 
$\Delta_{2}^{0}$ categorical.
\end{proof}

\begin{theorem}
\label{rdc2} If $\mathcal{A}$ is a computable equivalence structure with
finitely many infinite equivalence classes, then $\mathcal{A}$ is relatively 
$\Delta^0_2$ categorical.
\end{theorem}

\begin{proof}
There is a $\Sigma _{1}^{0}$ Scott formula for each element with an infinite
equivalence class, by the proof of Proposition \ref{p2}. There is a $\Sigma
_{1}^{0}$ Scott formula for each element with a finite equivalence class, by
the proof of Theorem \ref{rdc1}. It now follows, as before, that $\mathcal{A}
$ is relatively $\Delta _{2}^{0}$ categorical.
\end{proof}

\bigskip 

This leads to a stronger result for structures $\A$ with $Fin^{\A}$
computable.

\begin{theorem} \label{nt1} For any two isomorphic computable equivalence
  structures $\A$ and $\B$ such that $Fin^{\A}$ and $Fin^{\B}$ are
  both computable, $\A$ and $\B$ are $\Delta^0_2$ categorical. 
\end{theorem}

\begin{proof} It follows from Proposition \ref{p1} that $Inf^{\A}$ and
  $Inf^{\B}$ are computably isomorphic and it follows from Theorem
  \ref{rdc2} that $Fin^{\A}$ and $Fin^{\B}$ are $\Delta^0_2$
  isomorphic. Now the two mappings may be combined to define a
  $\Delta^0_2$ isomorphism between $\A$ and $\B$ since $Fin^{\A}$ and
  $Fin^{\B}$ are computable. 
\end{proof}

In fact, we observe that this result still holds if we only assume
that $Fin^{\A}$ and $Fin^{\B}$ are both $\Delta^0_2$. $Inf^{\A}$ and
$Inf^{\B}$ are still $\Delta^0_2$ isomorphic and there is a
$\Delta^0_2$ enumeration of the finite equivalence classes of $\A$ and
$\B$ which will induce a $\Delta^0_2$ isomorphism between $Fin^{\A}$
and $Fin^{\B}$. 

\begin{theorem} Let $\mathcal{A}$ be a computable equivalence
  structure with infinitely many infinite equivalence classes, and
  with unbounded character which has a computable $s_1$-function.
  Then $\mathcal{A}$ is not $\Delta^0_2$-categorical. \end{theorem}

\begin{proof}
We will build $\mathcal{B}_1 \simeq \mathcal{B}_2 \simeq \mathcal{A}$
in such a way as to diagonalize against $\Delta^0_2$ isomorphisms
between $\mathcal{B}_1$ and $\mathcal{B}_2$.  Let
$\varphi^{\Delta^0_2}_e$ be a computable enumeration of all partial
functions which are computable with a $\Delta^0_2$ oracle.  We will seek to meet the
following requirements:

\begin{tabular}{rl}
$R_e$: & $\forall x\ \varphi_e^{\Delta^0_2}\downarrow \Rightarrow
  \exists x_e \left| \left[
  \varphi_e^{\Delta^0_2}(x_e)\right]^{\mathcal{B}_2}\right| \neq
  \left| \left[ x_e \right]^{\mathcal{B}_1} \right|$\\
\end{tabular}

We will construct $\mathcal{B}_i$ exactly as in Lemma \ref{l2a}, with
the following exceptions.  We begin by partitioning $\omega$ into
two infinite disjoint parts.  One will provide ordinary
elements for the domain.  The other we will enumerate by $\{x_e\}_{e
  \in \omega}$, and so $x_e$ will serve as a witness for $R_e$.  Also,
in place of the elements $(k,n,w,z)$ elements, we will use elements of
the form $(k,n,w,z,q)$, where all except the last coordinate work just
as before.  Until permission is given to use elements where
$q>\tilde{q}$, no such elements will be used by the construction.

At stage $s$, we say that $R_e$ requires attention if $x_e \leq s$ and $\forall x \leq
s\ \varphi_e^{\Delta^0_2}(x)\downarrow \wedge \left| \left[
  \varphi_e^{\Delta^0_2}(x_e)\right]^{\mathcal{B}_2}\right| \neq
  \left| \left[ x_e \right]^{\mathcal{B}_1} \right|$.  For every
  $e$ that requires attention at stage $s$, we will act in the
  following way.  We will want to arrange that $\left[
    \varphi_e^{\Delta^0_2}(x_e)\right]^{\mathcal{B}_2}$ is of some
  finite size, and that $\left[ x_e \right]^{\mathcal{B}_1}$ is
  larger.

To achieve this, we will add elements to $\left[
  \varphi_e^{\Delta^0_2}(x_e)\right]^{\mathcal{B}_2}$,
  following the $s_1$-function, as in Lemma \ref{l4}.  All other
  instructions about adding elements to $\left[
  \varphi_e^{\Delta^0_2}(x_e)\right]^{\mathcal{B}_2}$ will now
  be discarded.  If $\left[
  \varphi_e^{\Delta^0_2}(x_e)\right]^{\mathcal{B}_2}$ contained an
  element of the form $(k,n,w,z,q)$, we will give permission to the
  element $(k,n,w,z,q+1)$, and will catch it up to the point reached
  by $(k,n,w,z,q)$.

In $\mathcal{B}_1$, we will choose a $j$ larger than that used for the
input of the $s_1$-function for $\left[
  \varphi_e^{\Delta^0_2}(x_e)\right]^{\mathcal{B}_2}$.  Increase the
size of $\left[ x_e \right]^{\mathcal{B}_1}$, following the
$s_1$-function as in Lemma \ref{l4}.  If $\left[ x_e
  \right]^{\mathcal{B}_1}$ contained an
  element of the form $(k,n,w,z,q)$, we will give permission to the
  element $(k,n,w,z,q+1)$, and will catch it up to the point reached
  by $(k,n,w,z,q)$.

Now by the proofs of Lemmas \ref{l2a} and \ref{l4}, the structures
$\mathcal{B}_1$ and $\mathcal{B}_2$ are equivalence structures with
infinitely many infinite equivalence classes, and with character
$\chi(\mathcal{A})$.  It remains to show that they are not
$\Delta^0_2$ isomorphic.

Suppose that $\varphi_e^{\Delta^0_2}$ is an isomorphism from
$\mathcal{B}_1$ to $\mathcal{B}_2$.  Suppose first that $\left[ x_e
  \right]^{\mathcal{B}_1}$ is finite, of size $n$.  There was some stage $s$ at
which $\left| \left[ x_e \right]^{\mathcal{B}_1} \right| = n$, and at
which $R_e$ received attention.  This stage prevented
$\varphi_e^{\Delta^0_2}$ from being an isomorphism.

Now suppose that $\left[ x_e \right]^{\mathcal{B}_1}$ is infinite.
Now $R_e$ received attention infinitely often, which could only happen
if $\varphi_e^{\Delta^0_2}(x_e) \uparrow$.
\end{proof}

It remains to consider the case of an unbounded character $K$ with no
$s_{1}$-function. Recall from Lemma \ref{l2a} that we may construct a
computable equivalence structure $\mathcal{A}$ with
$Fin^{\mathcal{A}}$ a $\Pi _{1}^{0}$ set. If we could also construct a
computable equivalence structure $\mathcal{B}$ with
$Fin^{\mathcal{B}}$ not a $\Delta _{2}^{0}$ set, then it would follow
that $\mathcal{A}$ is not $\Delta _{2}^{0}$ isomorphic to $\mathcal{A}
$. Surprisingly, we cannot make $Fin^{\mathcal{B}}$ a complete
$\Sigma_{2}^{0}$ set, as we could when $K$ had an
$s_{1}$-function. This is because of the following result.

\begin{theorem}
\label{nc1} Let $\mathcal{A}$ be a computable equivalence structure and let
$C$ be an infinite c.e. subset of $Fin^{\mathcal{A}}$. Then
$\mathcal{A}$ possesses a computable $s$-function $f$. Furthermore, if
$\{card([c]):c\in C\}$ is unbounded, then $\mathcal{A}$ possesses a
computable $s_{1}$-function $f$. Thus, there is a computable structure
with character $\chi (\mathcal{A}) $ and with no infinite equivalence classes.
\end{theorem}

\begin{proof}
Let $\mathcal{A}=(\omega ,E)$ where $E$ is a computable relation. Let 
\begin{equation*}
C_{1}=\{a:(\exists c\in C)cEa\}.
\end{equation*}
Then $C_{1}$ is also an infinite c.e. set. Fix its computable enumeration 
$\{c_{0},c_{1},\dots \}$ of $C_{1}$, without repetition. Now let 
\begin{equation*}
f(i,s)=card(\{x\leq s:xEc_{i}\})\text{.}
\end{equation*}
Let $\mathcal{A}_{1}=(\omega ,E_{1})$ where $iE_{1}j$ if and only if 
$c_{i}Ec_{j}$ and let $K_{1}=\chi (\mathcal{A}_{1})$. Then $f$ is clearly an 
$s$-function for $K_{1}$, and $K_{1}\subset \chi (\mathcal{A})$. It follows
from Lemma \ref{l4} that there is a computable structure with character 
$\chi (\mathcal{A})$ and with no infinite equivalence classes.

Now, suppose that $\{card([c]):c\in C\}$ is unbounded. Then $\mathcal{A}_{1}$
has unbounded character and no infinite equivalence classes. Thus, $K_{1}$
possesses an $s_{1}$-function $g$ by Lemma \ref{l3} and hence $K$ possesses
the same $s_{1}$-function.
\end{proof}

\begin{theorem}
\label{nc2} If the unbounded character $K$ has no $s_1$-function, then there
is no computable equivalence structure $\mathcal{A}$ with character $K$ such
that $Fin^{\mathcal{A}}$ is $\Sigma^0_2$ complete, or even $\Sigma^0_1$ hard.
\end{theorem}

\begin{proof}
Let $M$ be a complete c.e. set and suppose that there were a computable
function $f$ such that 
\begin{equation*}
i\in M\iff f(i)\in Fin^{\mathcal{A}}\text{.}
\end{equation*}
Then $C=\{f(i):i\in M\}$ is a c.e. subset of $Fin^{\mathcal{A}}$. If $C$ is
finite, say $C=\{c_{1},\dots ,c_{t}\}$, then 
\begin{equation*}
i\in M\iff (f(i)=c_{1}\vee f(i)=c_{2}\vee \dots \vee f(i)=c_{t})\text{,}
\end{equation*}
so that $M$ would be a computable set. Thus, $C$ is infinite. Now suppose
that $\{card([c]):c\in C\}$ is bounded by some finite $k$. Then $C$ is a
subset of the $\Pi _{1}^{0}$ set $P$, where 
\begin{equation*}
P=\{a:card([a])\leq k\}\text{.}
\end{equation*}
Since we have 
\begin{equation*}
i\in M\iff f(i)\in P\text{,}
\end{equation*}
that would imply that $M$ is a $\Pi _{1}^{0}$ set. This contradiction shows
that $\{card([c]):c\in C\}$ is unbounded. It now follows by Theorem \ref{nc1}
that $K$ possesses an $s_{1}$-function.
\end{proof}

\bigskip 

\textbf{Open Question}: Let $\mathcal{A}$ be a computable equivalence
structure having unbounded character, infinitely many infinite
equivalence classes, and no $s_1$-function. Further assume that
$Fin^{\mathcal{A}}$ is Turing incomparable with $\emptyset^\prime$.
Does such a structure exist, and if so, is this structure $\Delta
_{2}^{0}$ categorical?

\bigskip 

Theorem \ref{nc2} may provide some evidence that such a structure, if
it exists, may in fact be $\Delta^0_2$ categorical. Nevertheless, we
can show that such a structure cannot be \emph{relatively}
$\Delta^0_2$ categorical.

\begin{theorem}
If the computable equivalence structure $\mathcal{A}$ has unbounded
character and infinitely many infinite equivalence classes, then it is
not relatively $\Delta^0_2$ categorical.
\end{theorem}

\begin{proof}
Suppose, on the contrary, that an element $a$ with an infinite equivalence
class had a $\Sigma _{2}^{0}$ Scott formula $\psi (x,\overrightarrow{d})$.
Since there are only finitely many parameters $\overrightarrow{d}$ involved,
we may assume that $[a]$ does not contain any of the parameters 
$\overrightarrow{d}$. (This is where we use the assumption that there are
infinitely many infinite equivalence classes.) Then by choosing elements
$c_{1},\dots ,c_{n}$ of $\mathcal{A}$ to instantiate the existentially
quantified variables in $\psi (x,\overrightarrow{d})$, we would have a
computable $\Pi _{1}^{0}$ formula $\theta (x,\overrightarrow{d},
\overrightarrow{c})$, where $\overrightarrow{c}=c_{1},\dots ,c_{n}$, that is
satisfied by $a$.

Now, it is easy to see that for \emph{any} submodel $\mathcal{M}$ of
$\mathcal{A}$ that contains $a$, $\overrightarrow{d}$ and
$\overrightarrow{c}$, we have $\mathcal{M}\models \theta
(a,\overrightarrow{d},\overrightarrow{c} )$. In fact, since our
structures are relational, $\mathcal{A}\models \theta
(b,\overrightarrow{d},\overrightarrow{e})$ iff and only if
$\mathcal{M} \models \theta
(b,\overrightarrow{d},\overrightarrow{e})$ for all finite submodels
$\mathcal{M}$ of $\mathcal{A}$ that contain $b$, $\overrightarrow{d}$
and $\overrightarrow{e}$.

Thus, in particular, for the finite subset
$C=\{a,\overrightarrow{c},\overrightarrow{d}\}$ of $\omega $, we have
$\mathcal{C}=(C,E^{C})\models \theta
(a,\overrightarrow{d},\overrightarrow{c})$. Suppose that
$\overrightarrow{c}$ contains $m\leq n$ elements of $[a]$ and choose
$b$ such that $[b]\cap C=\emptyset $ and $m<card([b])<\omega $. (Here
we use the fact that the character of $\mathcal{A}$ is unbounded.) Let
$\mathcal{B}\simeq \mathcal{C}$ contain $m$ elements of $[b]$
(including $b$), together with $ C\setminus \lbrack a]$. Let
$\overrightarrow{e}$ denote the image of $\overrightarrow{c}$ under
the isomorphism between $\mathcal{C}$ and $\mathcal{B}$. Then
$\mathcal{B} \models \theta(b,\overrightarrow{d},\overrightarrow{e})$. Furthermore, let
$\mathcal{B}^{\prime }$ be any finite submodel of $\mathcal{A}$ such
that $\mathcal{B}\subseteq \mathcal{B}^{\prime }$. Then it is easy
to extend $\mathcal{C}$ to a finite submodel $\mathcal{C}^{\prime }$
that is isomorphic to $\mathcal{B}^{\prime }$ (where the isomorphism
fixes $\overrightarrow{d}$ pointwise and takes $a$ to $b$).  Thus,
$\mathcal{B}^{\prime }\models \theta (b,\overrightarrow{d},
\overrightarrow{e})$ as well. It follows that $\mathcal{A}\models
\theta (b, \overrightarrow{d},\overrightarrow{e})$. Hence,
$\mathcal{A}\models \psi (b, \overrightarrow{d})$.

But there certainly can be no automorphism of $\mathcal{A}$ mapping
$a$ to $b $, since $[a]$ is infinite and $[b]$ is finite. Thus, in
fact, $a$ cannot have a $\Sigma_{2}^{0}$ Scott formula.
\end{proof}

\bigskip

We conclude this section by looking at $\Delta _{3}^{0}$ categoricity.

\begin{theorem}
Every computable equivalence structure is relatively $\Delta _{3}^{0}$
categorical.
\end{theorem}

\begin{proof}
Any element with an infinite equivalence class has a $\Pi _{2}^{0}$ Scott
formula, while the other elements even have $\Delta _{2}^{0}$ Scott
formulas. Thus, every tuple $\langle a_{1},\dots ,a_{n}\rangle $ has a 
$\Sigma _{3}^{0}$ Scott formula.
\end{proof}


\begin{thebibliography}{99}
\bibitem{A} C.J. Ash, \textquotedblleft Categoricity in hyperarithmetical
degrees,\textquotedblright\ \emph{Annals of Pure and Applied Logic} 34
(1987), pp.\ 1--14.

\bibitem{A-K} C.J. Ash and J. F. Knight, \emph{Computable Structures and the
Hyperarithmetical Hierarchy} (Elsevier, Amsterdam, 2000).

\bibitem{A-K-M-S} C. Ash, J. Knight, M. Manasse and T. Slaman,
\textquotedblleft Generic copies of countable structures,\textquotedblright\ 
\emph{Annals of Pure and Applied Logic} 42 (1989), pp.\ 195--205.

\bibitem{A-N} C.J. Ash and A. Nerode, \textquotedblleft Intrinsically
recursive relations,\textquotedblright\ in \emph{Aspects of Effective Algebra
}, ed.\ by J. N. Crossley, Upside Down A Book Co., Steel's Creek, Australia,
1981, pp.\ 26--41.

\bibitem{CR} D. Cenzer and J. Remmel, \textquotedblleft Polynomial-time
Abelian groups,\textquotedblright\ \emph{Annals of Pure and Applied Logic}
56 (1992), pp.\ 313--363.

\bibitem{C} J. Chisholm, \textquotedblleft Effective model theory vs.\
recursive model theory,\textquotedblright\ \emph{Journal Symbolic Logic} 55
(1990), pp.\ 1168--1191.

\bibitem{C-G-K-S} P. Cholak, S. Goncharov, B. Khoussainov and R.A. Shore,
\textquotedblleft Computably categorical structures and expansions by
constants,\textquotedblright\ \emph{Journal\ of Symbolic Logic} 64 (1999),
pp.\ 13--37.

\bibitem{D} R.G. Downey, \textquotedblleft Computability theory and linear
orderings,\textquotedblright\ in: Yu. L. Ershov, S.S. Goncharov, A. Nerode
and J.B. Remmel, editors, Handbook of Recursive Mathematics, vol. 2
(North-Holland, Amsterdam, 1998), pp. 823--976.

\bibitem{G1} S.S. Goncharov, \textquotedblleft Autostability of models and
abelian groups,\textquotedblright\ \emph{Algebra and Logic} 19 (1980), pp.\
23--44 (Russian), pp. 13--27 (English translation).

\bibitem{G2} S.S. Goncharov, \textquotedblleft The quantity of
non-autoequivalent constructivizations,\textquotedblright\ \emph{Algebra and
Logic} 16 (1977), pp.\ 257--282 (Russian), pp. 169--185 (English
translation).

\bibitem{G} S.S. Goncharov, \textquotedblleft Autostability and computable
families of constructivizations,\textquotedblright\ \emph{Algebra and Logic}
14 (1975), pp.\ 647--680 (Russian), pp. 392--409 (English translation).

\bibitem{GD} S.S. Goncharov and V.D. Dzgoev, \textquotedblleft Autostability
of models,\textquotedblright\ \emph{Algebra and Logic} 19 (1980), pp.\
45--58 (Russian), pp. 28--37 (English translation).

\bibitem{G-H-K-M-M-S} S.S. Goncharov, V. S. Harizanov, J. F. Knight, C.F.D.
McCoy, R. G. Miller, and R. Solomon, \textquotedblleft Enumerations in
computable structure theory,\textquotedblright\ to appear in the 
\emph{Annals of Pure and Applied Logic}.

\bibitem{G-H-K-S} S.S. Goncharov, V.S. Harizanov, J.F. Knight and R.A.
Shore, \textquotedblleft $\Pi _{1}^{1}$ relations and paths through $%
\mathcal{O}$,\textquotedblright\ \emph{Journal of Symbolic Logic} 69 (2004),
pp. 585--611.

\bibitem{G-L-S} S. Goncharov, S. Lempp and R. Solomon, \textquotedblleft The
computable dimension of ordered abelian groups,\textquotedblright\ \emph{%
Advances in Mathematics} 175 (2003), pp. 102--143.

\bibitem{khis92} N.G. Khisamiev, \textquotedblleft Constructive Abelian $p$%
-groups \textquotedblright \emph{Siberian Advances in Mathematics} 2 (1992),
pp. 68--113

\bibitem{Kh} N.G. Khisamiev, \textquotedblleft Constructive Abelian
groups,\textquotedblright\ in: Yu. L. Ershov, S.S. Goncharov, A. Nerode and
J.B. Remmel, editors, Handbook of Recursive Mathematics, vol. 2
(North-Holland, Amsterdam, 1998), pp. 1177--1231.

\bibitem{K-S} B. Khoussainov and R.A. Shore, \textquotedblleft Computable
isomorphisms, degree spectra of relations and Scott
families,\textquotedblright\ \emph{Annals of Pure and Applied Logic} 93
(1998), pp.\ 153--193.

\bibitem{Ku} O.V. Kudinov, \textquotedblleft An autostable $1$-decidable
model without a computable Scott family of $\exists $-formulas,%
\textquotedblright\ \emph{Algebra and Logic} 35 (1996), pp.\ 458--467
(Russian), pp. 255--260 (English translation).

\bibitem{K1} O.V. Kudinov, \textquotedblleft A description of autostable
models,\textquotedblright\ \emph{Algebra and Logic} 36 (1997), pp.\ 26--36
(Russian), pp. 16--22 (English translation).

\bibitem{L} P. LaRoche, \textquotedblleft Recursively presented Boolean
algebras,\textquotedblright\ \emph{Notices AMS} 24 (1977), A552--A553.

\bibitem{L-M-M-S} S. Lempp, C.F.D. McCoy, R.G. Miller and D. R. Solomon,
\textquotedblleft Computable categoricity of trees of finite
height,\textquotedblright\ \emph{Journal of Symbolic Logic} 70 (2005), pp.
151--215.

\bibitem{McCoy} C.F.D. McCoy, \textquotedblleft $\Delta _{2}^{0}$%
-categoricity in Boolean algebras and linear orderings,\textquotedblright\ 
\emph{Annals of Pure and Applied Logic} 119 (2003), pp.\ 85--120.

\bibitem{MN} G. Metakides and A. Nerode, \textquotedblleft Effective content
of field theory,\textquotedblright\ \emph{Annals of Mathematical Logic} 17
(1979), pp.\ 289--320.

\bibitem{Millar} T. Millar, \textquotedblleft Recursive categoricity and
persistence,\textquotedblright\ \emph{Journal of Symbolic Logic} 51 (1986),
pp. \ 430--434.

\bibitem{Miller} R. Miller, \textquotedblleft The computable dimension of
trees of infinite height,\textquotedblright\ to appear in the \emph{Journal
of Symbolic Logic}.

\bibitem{N} A.T. Nurtazin, \textquotedblleft Strong and weak
constructivizations and computable families,\textquotedblright\ \emph{%
Algebra and Logic} 13 (1974), pp.\ 311--323 (Russian), pp. 177--184 (English
translation).

\bibitem{R} J.B. Remmel, \textquotedblleft Recursively categorical linear
orderings,\textquotedblright\ \emph{Proceedings of the American Mathematical
Society}$\mathbb{\ }$83 (1981), pp. 387--391.

\bibitem{S} V.L. Selivanov, \textquotedblleft Numerations of families of
general recursive functions,\textquotedblright\ \emph{Algebra and Logic} 15
(1976), pp.\ 205--226 (Russian), pp.\ 128--141 (English translation).

\bibitem{Sm} R.L. Smith, \textquotedblleft Two theorems on autostability in
p-groups,\textquotedblright\ in \emph{Logic Year 1979--80, Univ.
Connecticut, Storrs}, Lecture Notes in Mathematics 859, Springer, Berlin
(1981), pp. 302--311.

\bibitem{Soa87} R.I. Soare, \emph{Recursively Enumerable Sets and Degrees. A
Study of Computable Functions and Computably Generated Sets}
(Springer-Verlag, Berlin, 1987).
\end{thebibliography}
\end{document}